\documentclass[12pt, reqno]{amsart}
\usepackage{amsmath, amsthm, amscd, amsfonts, amssymb, graphicx, color}
\usepackage[bookmarksnumbered, colorlinks, plainpages]{hyperref}
\hypersetup{colorlinks=true,linkcolor=red, anchorcolor=green, citecolor=cyan, urlcolor=red, filecolor=magenta, pdftoolbar=true}

\usepackage{enumerate}
\usepackage{amssymb}
\usepackage{pst-node}
\usepackage{tikz-cd}

\newtheorem{theorem}{Theorem}[section]
\newtheorem{lemma}[theorem]{Lemma}
\newtheorem{proposition}[theorem]{Proposition}
\newtheorem{corollary}[theorem]{Corollary}
\theoremstyle{definition}

\theoremstyle{remark}
\newtheorem{remark}[theorem]{Remark}
\numberwithin{equation}{section}

\begin{document}

\setcounter{page}{1}

\title[Characterizing projections]{Characterizing projections among positive operators in the unit sphere}

\author[A.M. Peralta]{Antonio M. Peralta$^1$$^{*}$}

\address{$^{1}$Departamento de An{\'a}lisis Matem{\'a}tico, Facultad de
Ciencias, Universidad de Granada, 18071 Granada, Spain.}
\email{\textcolor[rgb]{0.00,0.00,0.84}{aperalta@ugr.es}}


\let\thefootnote\relax\footnote{Copyright 2016 by the Tusi Mathematical Research Group.}

\subjclass[2010]{Primary 47A05; Secondary 47L30, 46L05.}

\keywords{Projections; unit sphere around a subset; bounded linear operators; compact linear operators}

\date{Received: xxxxxx; Accepted: zzzzzz.
\newline \indent $^{*}$Corresponding author}

\begin{abstract} Let $E$ and $P$ be subsets of a Banach space $X$, and let us define the unit sphere around $E$ in $P$ as the set $$Sph(E;P) :=\left\{ x\in P : \|x-b\|=1 \hbox{ for all } b\in E \right\}.$$ Given a C$^*$-algebra $A$, and a subset $E\subset A,$ we shall write $Sph^+ (E)$ or $Sph_A^+ (E)$ for the set $Sph(E;S(A^+)),$ where $S(A^+)$ stands for the set of all positive operators in the unit sphere of $A$. We prove that, for an arbitrary complex Hilbert space $H$, the following statements are equivalent for every positive element $a$ in the unit sphere of $B(H)$:\begin{enumerate}[$(a)$]\item $a$ is a projection;
\item $Sph^+_{B(H)} \left( Sph^+_{B(H)}(\{a\}) \right) =\{a\}$.
\end{enumerate} We also prove that the equivalence remains true when $B(H)$ is replaced with an atomic von Neumann algebra or with $K(H_2)$, where $H_2$ is an infinite-dimensional and separable complex Hilbert space. In the setting of compact operators we prove a stronger conclusion by showing that the identity $$Sph^+_{K(H_2)} \left( Sph^+_{K(H_2)}(a) \right) =\left\{ b\in S(K(H_2)^+) : \!\! \begin{array}{c}
    s_{_{K(H_2)}} (a) \leq s_{_{K(H_2)}} (b),  \hbox{ and }\\
     \textbf{1}-r_{_{B(H_2)}}(a)\leq \textbf{1}-r_{_{B(H_2)}}(b)
  \end{array}\!\!
 \right\},$$ holds for every $a$ in the unit sphere of $K(H_2)^+$, where $r_{_{B(H_2)}}(a)$ and $s_{_{K(H_2)}} (a)$ stand for the range and support projections of $a$ in $B(H_2)$ and $K(H_2)$, respectively.
\end{abstract}

\maketitle

\section{Introduction}

In a recent attempt to solve a variant of Tingley's problem for surjective isometries of the
set formed by all positive operators in the unit sphere of $M_n(\mathbb{C})$, the space of all $n\times n$ complex matrices endowed with the spectral norm, G. Nagy has established an interesting characterization of those positive norm-one elements in $M_n(\mathbb{C})$ which are projections (see the final paragraph in the proof of \cite[Claim 1]{Nagy2017}). Motivated by the terminology employed by Nagy in the just quoted paper, we introduce here the notion of \emph{unit sphere around a subset} in a Banach space. Let $E$ and $P$ be subsets of a Banach space $X$. We define the \emph{unit sphere around $E$ in $P$} as the set $$Sph(E;P) :=\left\{ x\in P : \|x-b\|=1 \hbox{ for all } b\in E \right\}.$$ If $x$ is an element in $X$, we write $Sph(x;P)$ for $Sph(\{x\};P)$. Henceforth, given a Banach space $X$, let $S(X)$ denote the unit sphere of $X$. The cone of positive elements in a C$^*$-algebra $A$ will be denoted by $A^+$. If $M$ is a subset of $X$, we shall write $S(M)$ for $M\cap S(X)$. To simplify the notation, given a C$^*$-algebra $A$, and a subset $E\subset A,$ we shall write $Sph^+ (E)$ or $Sph_A^+ (E)$ for the set $Sph(E;S(A^+))$. For each element $a$ in $A$, we shall write $Sph^+ (a)$ instead of $Sph^+ (\{a\})$. \smallskip

Let $a$ be a positive norm-one element in $B(\ell_2^n)=M_n(\mathbb{C})$. The commented characterization established by Nagy proves that the following two statements are equivalent:
\begin{equation}\label{eq bi sphere around a point}
\begin{array}{l}
          \hbox{$(i)$\ $a$ is a projection}\ \ \ \ \ \ \ \ \ \ \ \ \ \ \ \ \ \ \ \ \ \ \ \ \ \ \ \ \ \ \ \ \ \ \ \ \ \ \ \ \ \ \ \ \ \ \ \ \ \ \ \ \ \ \ \ \ \ \ \ \ \ \ \ \ \ \  \\
          (ii) \ Sph^+_{M_n(\mathbb{C})} \left( Sph^+_{M_n(\mathbb{C})}(a) \right) =\left\{ a \right\},
        \end{array}
 \end{equation} (see the final paragraph in the proof of \cite[Claim 1]{Nagy2017}). As remarked by G. Nagy in \cite[\S 3]{Nagy2017}, the previous characterization (and the whole statement in \cite[Claim 1]{Nagy2017}) remains as an open problem when $H$ is an arbitrary complex Hilbert space. This is an interesting problem to be considered in operator theory, and in the wider setting of general C$^*$-algebras.\smallskip

In this note we extend the characterization in \eqref{eq bi sphere around a point} to the case in which $H$ is an arbitrary complex Hilbert space. In a first result we prove that, for any positive element $a$ in the unit sphere of a C$^*$-algebra $A$, the equality $Sph^+_{A} \left( Sph^+_{A}(a) \right) =\{a\}$  is a sufficient condition to guarantee that $a$ is a projection in $A$ (cf. Proposition \ref{p necessary condition for projection in terms of the sphere around a positive element}). In Theorem \ref{t characterization of projection in terms of the sphere around a positive element} we extend Nagy's characterization to the setting of atomic von Neumann algebras by showing that the following statements are equivalent for every positive norm-one element $a$ in
an atomic von Neumann algebra $M$ (in particular when $M=B(H)$, where $H$ is an arbitrary complex Hilbert space): \begin{enumerate}
[$(a)$] \item $a$ is a projection; \item $Sph^+_{M} \left( Sph^+_{M}(a) \right) =\{a\}$.
\end{enumerate}

We shall also explore whether the above characterization also holds when $M$ is replaced with $K(H)$, the space of all compact operators on a complex Hilbert space $H$. Our conclusion in this case is the following: Let $H_2$ be a separable complex Hilbert space, and suppose that $a$ is a positive norm-one element in $K(H_2)$. Then the following statements are equivalent: \begin{enumerate}
[$(a)$] \item $a$ is a projection; \item $Sph^+_{K(H_2)} \left( Sph^+_{K(H_2)}(a) \right) =\{a\}$.
\end{enumerate}

When $H$ is a finite-dimensional complex Hilbert space Nagy computed in \cite{Nagy2017} the second unit sphere around a positive element in the unit sphere of $B(H)^+$, and showed that the identity $$ Sph^+_{B(H)} \left( Sph^+_{B(H)}(a) \right) =\left\{ b\in S(B(H)^+) : \begin{array}{c}
                                                                                           \hbox{Fix}(a) \subseteq \hbox{Fix} (b), \\
                                                                                            \hbox{ and } \ker(a)\subseteq \ker(b)
                                                                                         \end{array} \right\}$$
holds for every element $a$ in $S(B(H)^+),$ where for each $a$ in $S(B(H)^+)$ we set $\hbox{Fix}(a)=\{\xi\in H : a(\xi) = \xi\},$ (see the beginning of the proof of \cite[Claim 1]{Nagy2017}). In Theorem \ref{t bi spherical set K(H)} we establish a generalization of this fact to he setting of compact operators. We prove that if $H_2$ is a separable infinite-dimensional complex Hilbert space, then the identity $$Sph^+_{K(H_2)} \left( Sph^+_{K(H_2)}(a) \right) =\left\{ b\in S(K(H_2)^+) : \!\! \begin{array}{c}
    s_{_{K(H_2)}} (a) \leq s_{_{K(H_2)}} (b),  \hbox{ and }\\
     \textbf{1}-r_{_{B(H_2)}}(a)\leq \textbf{1}-r_{_{B(H_2)}}(b)
  \end{array}\!\!
 \right\},$$ holds for every $a$ in the unit sphere of $K(H_2)^+$, where $r_{_{B(H_2)}}(a)$ and $s_{_{K(H_2)}} (a)$ stand for the range and support projections of $a$ in $B(H_2)$ and $K(H_2)$, respectively.\smallskip

As we have already commented at the beginning of this introduction, the characterization obtained by Nagy in \eqref{eq bi sphere around a point} is one of the key results to establish that every surjective isometry $\Delta: S(M_n(\mathbb{C})^+) \to S(M_n(\mathbb{C})^+)$ admits an extension to a surjective real linear or complex linear isometry on $M_n(\mathbb{C})$ (see \cite[Theorem]{Nagy2017}). Another related results are known when $M_n(\mathbb{C})= B(\ell_2^n)$ is replaced with the space $(C_p(H), \|\cdot\|_p)$ of all $p$-Schatten-von Neumann operators on a complex Hilbert space $H$, with $1\leq p< \infty$. L. Moln{\'a}r and W. Timmermann proved that for every complex Hilbert space $H$, every surjective isometry $\Delta : S(C_1(H)^+) \to S(C_1(H)^+)$ can be extended to a surjective complex linear isometry on $C_1(H)$. Nagy shows in \cite[Theorem 1]{Nagy2013} that the same conclusion remains true for every $1<p <\infty$.\smallskip

The results commented in the previous paragraph are subtle variants of the so-called Tingley's problem. This problem asks whether every surjective isometry between the unit spheres of two Banach spaces $X$ and $Y$ admits an extension to a surjective real linear isometry from $X$ onto $Y$. Tingley's problem remains open after thirty years. However, in what concerns operator algebras, certain positive solutions to this problem have been recently established in the setting of finite-dimensional C$^*$-algebras and finite von Neumann algebras \cite{Tan2017, Tan2017b}, spaces of compact linear operators and compact C$^*$-algebras \cite{PeTan16}, $B(H)$ spaces \cite{FerPe17b}, von Neumann algebras \cite{FerPe17d}, spaces of trace class operators \cite{FerGarPeVill17}, preduals of von Neumann algebras \cite{Mori2017}, and spaces of $p$-Schatten von Neumann operators with $2<p<\infty$ \cite{FerJorPer2018}. The reader is referred to the survey \cite{Pe2018} for additional details.\smallskip

After completing the description of all surjective isometries on $S(M_n(\mathbb{C})^+)$, Nagy conjectured that a similar result should also hold for surjective  surjective isometries on $S(B(H)^+)$, where $H$ is an arbitrary complex Hilbert space (see \cite[\S 3]{Nagy2017}). The results presented in this note are a first step towards a proof of Nagy's conjecture.

\section{The results}\label{sec: characterization of projections}

Let us fix some notation. Along the paper, the closed unit ball and the dual space of a Banach space $X$ will be denoted by $\mathcal{B}_{_X}$ and $X^*$, respectively. Given a subset $B\subset X,$ we shall write $\mathcal{B}_{_B}$ for $\mathcal{B}_{_X}\cap B$.\smallskip

The cone of positive elements in a C$^*$-algebra $A$ will be denoted by $A^+$, while the symbol  $(A^*)^+$ will stand for the set of positive functionals on $A$. A \emph{state} of $A$ is a positive functional in $S(A^*)$. The set of states of $A$ will be denoted by $\mathcal{S}_A$. It is well known that $\mathcal{B}_{_{(A^*)^+}} = \mathcal{B}_{_{A^*}}\cap (A^*)^+$ is a weak$^*$-closed convex subset of $\mathcal{B}_{_{A^*}}$. The set of \emph{pure states} of $A$ is precisely the set $\partial_e (\mathcal{B}_{_{(A^*)^+}})$ of all extreme points of $\mathcal{B}_{_{(A^*)^+}}$ (see \cite[\S 3.2]{Ped}).\smallskip



Suppose $a$ is a positive element in the unit sphere of a von Neumann algebra $M$. The \emph{range projection} of $a$ in $M$ (denoted by $r(a)$) is the smallest projection $p$ in $M$ satisfying $a p = a$. It is known that the sequence $\left( (1/n \textbf{1}+a)^{-1} a \right)_n$ is monotone increasing to $r(a)$, and hence it converges to $r(a)$ in the weak$^*$-topology of $M$. Actually, $r(a)$ also coincides with the weak$^*$-limit of the sequence $(a^{1/n})_{n}$ in $M$ (see \cite[2.2.7]{Ped}).
It is also known that the sequence $(a^{n})_{n}$ converges to a projection $s(a)=s_{_{M}}(a)$ in $M,$ which is called the \emph{support projection} of $a$ in $M$. Unfortunately, the support projection of a norm-one element in $M$ might be zero. For example, let $\{\xi_n : n\in \mathbb{N}\}$ denote an orthonormal basis of $\ell_2$, and let $a$ be the positive operator in $B(\ell_2)$ given by $\displaystyle a=\sum_{m=1}^{\infty} \frac{m-1}{m} p_m$, where, for each $m$, $p_m$ is the rank one projection $\xi_m\otimes \xi_m$. It is not hard to check that $s_{B(\ell_2)}(a) = 0$.\smallskip

Elements $a,b$ in a C$^*$-algebra $A$ are called orthogonal (written $a\perp b$) if $a b^* = b^* a =0$. It is known that $ \|a+ b\| =\max\{\|a\|,\|b\| \},$ for every $a,b\in A$ with $a\perp b$. Clearly, self-adjoint elements $a,b\in A$ are orthogonal if and only if $a b =0$.\smallskip

We recall some geometric properties of C$^*$-algebras. Let $p$ be a projection in a unital C$^*$-algebra $A$. Suppose that $x\in S(A)$ satisfies $pxp =p,$ then \begin{equation}\label{eq Peirce 2 norm one}\hbox{$x = p + (\textbf{1}-p) x (\textbf{1}-p)$, \ \ }
 \end{equation} (see, for example, \cite[Lemma 3.1]{FerPe18}).
Another property needed later reads as follows:  Suppose that $b\in A^+$ satisfies $pbp =0,$  then \begin{equation}\label{eq pbp =0 implies p perp b for p positive}\hbox{ $p b = b p =0$, equivalently, $p \perp b$.}
 \end{equation} To see this property let us take a positive $c\in A$ satisfying $c^2 = b$. The identity $0\leq (p c)(p c)^* = p c^2 p = p b p =0$ and the Gelfand-Naimark axiom imply that $ pc = c p=0,$ and hence $p b = p c^2 =0 = c^2 p = b p$. \smallskip

A non-zero projection $p$ in a C$^*$-algebra $A$ is called minimal if $p A p = \mathbb{C} p$. A von Neumann algebra $M$ is called atomic if it coincides with the weak$^*$ closure of the linear span of its minimal projections. It is known from the structure theory of von Neumann algebras that every atomic von Neumann algebra $M$ can be written in the form $\displaystyle M = \bigoplus_j^{\ell_{\infty}} B(H_{j}),$ where each $H_j$ is a complex Hilbert space (compare \cite[\S 2.2]{S} or  \cite[\S V.1]{Tak}).\smallskip

Let $p$ be a non-zero projection in an atomic von Neumann algebra $\displaystyle M=\bigoplus_j^{\ell_{\infty}} B(H_{j})$. In this case we can always find a family $(q_{\lambda})$ of mutually orthogonal minimal projections in $M$ such that $\displaystyle p = \hbox{w$^*$-}\sum_{\lambda} q_{\lambda}$ (compare \cite[Definition 1.13.4]{S}). Furthermore, $p$ is the least upper bound of the set of all minimal projections in $M$ which are smaller than or equal to $p$.\smallskip

The bidual, $A^{**}$, of a C$^*$-algebra $A$ is a von Neumann algebra whose predual contains an abundant collection of pure states of $A$. This geometric advantage implies that the support projection in $A^{**}$ of every element in $S(A^+)$ is a non-zero projection. Namely, if $a$ lies in $S(A^+)$ it is well known that we can find a pure state $\phi\in\partial_e (\mathcal{B}_{_{(A^*)^+}})$ satisfying $\phi (a) =1$. Pure states in $A^*$ are in one-to-one correspondence with minimal projections in $A^{**}$, more concretely, for each $\phi\in\partial_e (\mathcal{B}_{_{(A^*)^+}})$ there exists a unique minimal partial isometry $p_{\phi}\in A^{**}$ satisfying $\phi (p_{\phi}) =1,$ and $p_{\phi} x p_{\phi} = \phi (x) p_{\phi}$ for all $x\in M$ (see \cite[Proposition 3.13.6]{Ped}). The projection $p_{\phi}$ is called the \emph{support projection} of $\phi$. Since $A$ is weak$^*$-dense in $A^{**},$ and the product of the latter von Neumann algebra is separately weak$^*$-continuous (see \cite[Proposition 3.6.2 and Remark 3.6.5]{Ped} or \cite[Theorem 1.7.8]{S}), it can be easily seen that every minimal projection in $A$ is minimal in $A^{**}$.\smallskip

Let $a$ be a positive norm-one element in a C$^*$-algebra $A$. Let us take an state $\phi\in \mathcal{S}_{A}$ satisfying $\phi (a) =1$ (compare \cite[Proposition 1.5.4 and its proof]{S}). The set $\{\psi\in \mathcal{B}_{_{(A^*)^+}} : \psi (a) =1\}$ is a non-empty weak$^*$ closed convex subset of $\mathcal{B}_{_{A^*}}$. By the Krein-Milman theorem there exists $\varphi \in \partial_e(\mathcal{B}_{_{(A^*)^+}})$ belonging to the previous set, and hence $\varphi (a) =1$. We consider the support projection $p_{\varphi}$ of $\varphi$ in $A^{**}$, which is a minimal projection. The condition $\varphi (a)=1$ implies $p_{\varphi} = p_{\varphi} a p_{\varphi}$, and \eqref{eq Peirce 2 norm one} assures that $ a= p_{\varphi} + (\textbf{1}-p_{\varphi}) a(\textbf{1}-p_{\varphi})$, and thus $0\neq p_{\varphi}\leq s_{_{A^{**}}} (a)$. We can therefore deduce that \begin{equation}\label{eq support in bidual is non-null} s_{_{A^{**}}} (a) \neq 0, \hbox{ for all } a\in S(A^+).
\end{equation}

In order to recall the connections with Nagy's paper, we observe that, given a norm-one positive operator $a$ in $B(H)$, we denote $\hbox{Fix}(a)=\{\xi\in H : a(\xi) = \xi\},$ and we write $p_a$ for the projection of $H$ onto Fix$(a)$. Since $a = p_a + (\textbf{1}-p_a) a (\textbf{1}-p_a)$, it follows that $p_a$ is smaller than or equal to the support projection of $a$ in $B(H)^{**}$. In some cases, $p_a$ may be zero while $s_{_{B(H)^{**}}} (a)\neq 0$. When $H$ is finite dimensional $p_a$ and $s(a)$ coincide. If we take a positive norm-one element in the space $K(H)$ of all compact operators on $H$, the element $s_{_{B(H)}} (a)= s_{_{K(H)^{**}}} (a)=p_a$ is a (non-zero) finite rank projection and lies in $K(H)$. We shall write $s_{_{K(H)}} (a)$ for the projection $s_{_{B(H)}} (a)$.\smallskip

If $p$ is a non-zero projection in a C$^*$-algebra $A$ then \begin{equation}\label{eq new 0711} \hbox{for each $a$ in $S(A^+)$ such that $p\leq a$, we have $a = p + (\textbf{1}-p) a (\textbf{1}-p)$.}\end{equation} Namely, under the above hypothesis, we also have $p\leq a$ in the von Neumann algebra $A^{**}$. It follows that $p\leq s_{_{A^{**}}} (a)\leq a$, and hence $s_{_{A^{**}}} (a)-p$ is a projection in $A^{**}$ which is orthogonal to $p$. Since $a = s_{_{A^{**}}} (a) + (\textbf{1}- s_{_{A^{**}}} (a) ) a (\textbf{1}-s_{_{A^{**}}} (a))$, we have $p a p  = p s_{_{A^{**}}} (a) p = p$, and thus $a = p + (\textbf{1}-p) a (\textbf{1}-p)$ (compare \eqref{eq Peirce 2 norm one}). \smallskip

It is part of the folklore in the theory of C$^*$-algebras that the distance between two positive elements $a,b$ in the closed unit ball of a C$^*$-algebra $A$ is bounded by one. Namely, since $-\textbf{1}\leq -b \leq a-b \leq a \leq \textbf{1}$, we deduce that $\|a -b\| \leq 1$.\smallskip

In our first result, which is an infinite-dimensional version of \cite[Corollary]{Nagy2017}, we establish a precise description of those pairs of elements in $S(A^+)$ whose distance is exactly one.

\begin{lemma}\label{l when distance is exactly one} Let $A$ be a C$^*$-algebra, and let $a,b$ be elements in $S(A^+)$. Then $\|a-b\| = 1$ if and only if there exists a minimal projection $e$ in $A^{**}$ such that one of the following statements holds:
\begin{enumerate}[$(a)$]\item $e\leq a$ and $e\perp b$ in $A^{**}$;
\item $e\leq b$ and $e\perp a$ in $A^{**}$.
\end{enumerate}
\end{lemma}

\begin{proof} Let us first assume that  $\|a-b\| = 1$. Arguing as in the proof of \eqref{eq support in bidual is non-null}, we can find $\varphi \in \partial_e(\mathcal{B}_{_{(A^*)^+}})$ such that $|\varphi (a-b) |=1$. Since $0\leq \varphi (a), \varphi (b) \leq 1$, we can deduce that precisely one of the following holds: \begin{enumerate}[$(a)$]\item $\varphi (a)=1$ and $\varphi (b)=0$;
\item $\varphi (b)=1$ and $\varphi (a)=0$.
\end{enumerate} Let $e=p_{\varphi}$ be the minimal projection in $A^{**}$ associated to the pure state $\varphi$. In case $(a)$ we know that $e a e=e$ and $e be =0$. Thus, by \eqref{eq Peirce 2 norm one} and \eqref{eq pbp =0 implies p perp b for p positive} it follows that $a = e + (\textbf{1}-e) a (\textbf{1}-e)\geq e$ and $b \perp e$ in $A^{**}$. Similar arguments show that in case $(b)$ we get $e\leq b$ and $e\perp a$ in $A^{**}$.\smallskip

Suppose now that we can find a minimal projection $e$ in $A^{**}$ satisfying $(a)$ or $(b)$ in the statement of the lemma. We shall only consider the case in which statement $(a)$ holds, the other case is identical. Let $\varphi$ be the pure state in $A^*$ associated with $e$. Since $a = e + (\textbf{1}-e) a (\textbf{1}-e)$ and $b =  (\textbf{1}-e) b (\textbf{1}-e)$ in $A^{**}$ we obtain $\varphi (a-b )= \varphi (e) =1\leq \|a-b\| \leq 1.$
\end{proof}

We are now in position to establish a sufficient condition in terms of the set $Sph^+_{A} \left( Sph^+_{A}(a) \right)$, to guarantee that a positive norm-one element $a$ in a C$^*$-algebra $A$ is a projection.

\begin{proposition}\label{p necessary condition for projection in terms of the sphere around a positive element} Let $A$ be a C$^*$-algebra, and let $a$ be a positive norm-one element in $A$. Suppose $Sph^+_{A} \left( Sph^+_{A}(a) \right) =\{a\}$. Then $a$ is a projection.
\end{proposition}

\begin{proof} Let $\sigma(a)$ denote the spectrum of $a$. We identify the C$^*$-subalgebra of $A$ generated by $a$ with the commutative C$^*$-algebra $C_0(\sigma(a))$ of all continuous functions on $\sigma(a)\cup\{0\}$ vanishing at $0$. Fix an arbitrary function $c\in C_0(\sigma(a))$ with $0\leq c\leq 1$, $c(0) =0$ and $c(1) =1$. We claim that any such element $c$ satisfies the following properties: \begin{enumerate}[$(P1)$]\item If $q$ is a minimal projection in $A^{**}$ with $q\leq a$, then $q\leq c$ in $A^{**}$;
\item If $q$ is a projection in $A^{**},$ with $q\perp a =0$ then $q c=0$.
\end{enumerate}
We shall next prove the claim. $(P1)$ Let $q$ be a minimal projection in $A^{**}$ with $q\leq a.$ Let $\varphi\in \partial_e(\mathcal{B}_{_{(A^*)^+}})$ be a pure state of $A$ satisfying $\varphi (q)= 1$. In this case $a = q + (\textbf{1}-q) a(\textbf{1}-q)$ in $A^{**}$. This proves that $s_{_{A^{**}}} (a) = q + s_{_{A^{**}}} ((\textbf{1}-q) a(\textbf{1}-q)) \geq q$ in $A^{**}$. The element $c$ has been defined to satisfy $s_{_{C_0(\sigma(a))^{**}}} (a) \leq s_{_{C_0(\sigma(a))^{**}}} (c)$. Since $C_0(\sigma(a))^{**}$ can be identified with the weak$^*$ closure of $C_0(\sigma(a))^{**}$ in $A^{**}$, we can actually conclude that $q\leq s_{_{A^{**}}} (a) = s_{_{C_0(\sigma(a))^{**}}} (a) \leq s_{_{C_0(\sigma(a))^{**}}} (c)= s_{_{A^{**}}} (c)$. This implies that $\varphi (c) = 1$ and hence $q\leq c$ in $A^{**}$.\smallskip

$(P2)$ Any element in $A^{**}$ which is orthogonal to $a$ must be orthogonal to every element in $C_0(\sigma(a)),$ because the latter is the C$^*$-subalgebra of $A$ generated by $a$. This finishes the proof of the claim.\smallskip

By Lemma \ref{l when distance is exactly one}, an element $x$ lies in $Sph^+_{A}(a)$ if and only if there exists a minimal projection $e$ in $A^{**}$ such that one of the following statements holds:
\begin{enumerate}[$(a)$]\item $e\leq a$ and $e\perp x$ in $A^{**}$;
\item $e\leq x$ and $e\perp a$ in $A^{**}$.
\end{enumerate}

In case $(a)$, $e\perp x$ and $e\leq c$ by $(P1)$, and Lemma \ref{l when distance is exactly one} implies that $\|x-c\|=1$.\smallskip

In case $(b)$, $e\leq x$ and $e\perp a,$ and hence $e\perp c$ by $(P2)$. Lemma \ref{l when distance is exactly one} implies that $\|x-c\|=1$.\smallskip

We have proved that, any function $c\in C_0(\sigma(a))$ with $0\leq c\leq 1$, $c(0) =0$ and $c(1) =1$ belongs to $Sph^+_{A} \left( Sph^+_{A}(a) \right)=\{a\},$ which forces to $\sigma(a) =\{0,1\},$ and hence $a$ is a projection.
\end{proof}

The promised characterization of non-zero projections in an atomic von Neumann algebra is established next.

\begin{theorem}\label{t characterization of projection in terms of the sphere around a positive element} Let $M$ be an atomic von Neumann algebra, and let $a$ be a positive norm-one element in $M$. Then the following statements are equivalent: \begin{enumerate}
[$(a)$] \item $a$ is a projection; \item $Sph^+_{M} \left( Sph^+_{M}(a) \right) =\{a\}$.
\end{enumerate}
\end{theorem}

\begin{proof} $(a)\Rightarrow (b)$ Suppose $a=p$ is a projection. Clearly $$\{p\} \subseteq Sph^+_{M} \left( Sph^+_{M}(p) \right).$$ Let us take $b$ in the set $Sph^+_{M} \left( Sph^+_{M}(p) \right)$. We shall first prove that $\textbf{1}-p \perp b$. If $\textbf{1}-p =0$ there is nothing to prove. Otherwise, let $e$ be a minimal projection in $M$ with $e\leq \textbf{1}-p$. Since $\| e + \frac12 (\textbf{1}-e) -p\| =1,$ we deduce that $\| e + \frac12 (\textbf{1}-e) -b\| =1$.\smallskip

Lemma \ref{l when distance is exactly one} proves the existence of a minimal projection $q\in M^{**}$ such that one of the next statements holds: \begin{enumerate}[$(1)$]\item $q\leq e + \frac12 (\textbf{1}-e)$ and $q\perp b$ in $M^{**}$;
\item $q\leq b$ and $q\perp e + \frac12 (\textbf{1}-e)$ in $M^{**}$.
\end{enumerate} We claim that case $(2)$ is impossible. Indeed, $q\perp e + \frac12 (\textbf{1}-e)$ is equivalent to $q\perp r_{_{M^{**}}} (e + \frac12 (\textbf{1}-e)) =\textbf{1},$ which is impossible. Therefore, only case $(1)$ holds, and thus $q\leq e$. Since $e$ also is a minimal projection in $M^{**}$, we deduce from the minimality of $q$ that $e=q\perp b$.\smallskip

We have shown that for every minimal projection $e$ in $M$ with $e\leq \textbf{1}-p$ we have $e\perp b$. Since $\textbf{1}-p$ is the least upper bound of all minimal projections $q$ in $M$ with $q\leq \textbf{1}-p$ (actually $\displaystyle \textbf{1}-p =\sum_j e_j$ where $\{e_j\}$ is a family of mutually orthogonal minimal projections in $M$), it follows that $\textbf{1}-p \perp b$ (equivalently, $ p b = b p =b$).\smallskip

We shall next prove that $b$ is a projection and $p=b$. Let $\sigma(b)$ be the spectrum of $b$, let $\mathcal{C}$ denote the C$^*$-subalgebra of $M$ generated by $b$ and $p$, and let us identify $\mathcal{C}$ with $C(\sigma(b))$, $b$ with the function $t\mapsto t$, and $p$ with the unit of $\mathcal{C}$. We shall distinguish two cases:\begin{enumerate}[$(i)$] \item $0\notin \sigma(b)$ (that is, $b$ is invertible in $\mathcal{C}$); \item $0\in \sigma(b)$ (that is, $b$ is not invertible in $\mathcal{C}$).
\end{enumerate}

We deal first with case $(i)$. If $0\notin \sigma(b)$, let $m_0$ be the minimum of $\sigma(b)$. If $0<m_0<1$, we consider the function $d\in \mathcal{C}\equiv C(\sigma(b))$ defined by $d(t)=  \frac{1}{1-m_0} (t-m_0)$, ($t\in\sigma(b)$). It is not hard to check that $0\leq \| b-d \| =m_0<1 $ and $\|p-d \| =1,$ which contradicts that $b \in Sph^+_{M} \left( Sph^+_{M}(p) \right)$. Therefore $m_0=1,$ and hence $b$ is invertible with $\sigma(b) =\{1\}$, witnessing that $\textbf{1}=b\leq p \leq \textbf{1}$. We have proved that $b=p =\textbf{1}$.\smallskip

In case $(ii)$, $0\in \sigma(b)$. If there exists $t_0\in \sigma(b)\cap (0,1)$, the function \begin{equation}\label{eq function c} c(t)= \left\{\begin{array}{cc}
                                   0 & \hbox{ if } t\in \sigma(b)\cap [0,t_0]; \\
                                   \frac{1+t_0}{1-t_0} (t-t_0) & \hbox{ if } t\in \sigma(b)\cap [t_0,\frac{1+t_0}{2}]; \\
                                   t & \hbox{ if } t\in \sigma(b)\cap [\frac{1+t_0}{2},1],
                                 \end{array}
 \right.
 \end{equation} defines a positive norm-one element in $c\in C(\sigma(b))$ such that $\| p -c\| =1,$ and $\| b-c\| =t_0<1$. This contradicts that $b \in Sph^+_{M} \left( Sph^+_{M}(p) \right)$. Therefore, $\sigma(b) \subseteq\{0,1\}$, and hence $b$ is a projection. If $b < p$, we get $\| b - b\|=0$ and $\|p -b\| = 1$, contradicting that $b \in Sph^+_{M} \left( Sph^+_{M}(p) \right)$. Therefore $p=b$.\smallskip

 We have shown that $Sph^+_{M} \left( Sph^+_{M}(p) \right) =\{p\}$.\smallskip

The implication $(b)\Rightarrow (a)$ follows from Proposition \ref{p necessary condition for projection in terms of the sphere around a positive element}.
\end{proof}

The next result is a clear consequence of our previous theorem and extends the characterization of projections in $M_n(\mathbb{C})$ established by G. Nagy in the final paragraph of the proof of \cite[Claim 1]{Nagy2017} (compare \eqref{eq bi sphere around a point}).

\begin{corollary}\label{c characterization of projection in terms of the sphere around a positive element B(H)} Let $H$ be an arbitrary complex Hilbert space, and let $a$ be a positive norm-one element in $B(H)$. Then the following statements are equivalent: \begin{enumerate}
[$(a)$] \item $a$ is a projection; \item $Sph^+_{B(H)} \left( Sph^+_{B(H)}(a) \right) =\{a\}$. $\hfill\Box$
\end{enumerate}
\end{corollary}

It seems natural to ask whether the above corollary remains true if $B(H)$ is replaced with $K(H)$.
For an infinite-dimensional separable complex Hilbert space $H_2$, the conclusion of Theorem \ref{t characterization of projection in terms of the sphere around a positive element} and Corollary \ref{c characterization of projection in terms of the sphere around a positive element B(H)} can be also extended to projections in the space $K(H_2)$. The arguments in the proof of Theorem \ref{t characterization of projection in terms of the sphere around a positive element} actually require a subtle adaptation.

\begin{theorem}\label{t characterization of projection in terms of the sphere around a positive element k(H)} Let $a$ be a positive norm-one element in $K(H_2)$, where $H_2$ is a separable complex Hilbert space. Then the following statements are equivalent: \begin{enumerate}
[$(a)$] \item $a$ is a projection; \item $Sph^+_{K(H_2)} \left( Sph^+_{K(H_2)}(a) \right) =\{a\}$.
\end{enumerate}
\end{theorem}

\begin{proof} When $H_2$ is finite-dimensional the equivalence is proved in \cite[final paragraph of the proof of Claim 1]{Nagy2017}. We can therefore assume that $H_2$ is infinite-dimensional.\smallskip

$(a)\Rightarrow (b)$ We assume first that $a=p\in K(H_2)$ is a projection. We can find a family  $\{q_1,\ldots,q_n\}$ of mutually orthogonal minimal projections in $K(H)$ such that $\displaystyle p = \sum_{j=1}^n q_j$. As before, the inclusion $$\{p\} \subseteq Sph^+_{K(H_2)} \left( Sph^+_{K(H_2)}(p) \right)$$ always holds. Let us take $b$ in the set $Sph^+_{K(H_2)} \left( Sph^+_{K(H_2)}(p) \right)$. Clearly $0\neq \textbf{1}-p\notin K(H_2)$. Let $e$ be a minimal projection in $K(H_2)$ with $e\leq \textbf{1}-p$ in $B(H_2)$. Since $H_2$ is separable, we can pick a maximal family $\{ v_{n} : {n\in \mathbb{N}}\}$ of mutually orthogonal minimal projections in $(\textbf{1}-e) K(H_2) (\textbf{1}-e)$ with $\displaystyle \textbf{1}-e = \sum_{n=1}^\infty v_n$. The element $\displaystyle e + \sum_{n=1}^\infty \frac{1}{2n} v_n$ lies in $S(K(H_2)^+)$ and $\displaystyle \left\| e + \sum_{n=1}^\infty \frac{1}{2n} v_n -p\right\| =1,$ thus, the hypothesis on $b$ implies that $\displaystyle \left\| e + \sum_{n=1}^\infty \frac{1}{2n} v_n -b\right\| =1$.  Lemma \ref{l when distance is exactly one} proves the existence of a minimal projection $q\in K(H_2)^{**}=B(H_2)$ such that one of the next statements holds: \begin{enumerate}[$(1)$]\item $\displaystyle q\leq e + \sum_{n=1}^\infty \frac{1}{2n} v_n$ and $q\perp b$ in $K(H_2)^{**}=B(H_2)$;
\item $q\leq b$ and $\displaystyle q\perp e + \sum_{n=1}^\infty \frac{1}{2n} v_n$ in $K(H_2)^{**}=B(H_2)$.
\end{enumerate}

In case $(2)$, $\displaystyle q\perp e + \sum_{n=1}^\infty \frac{1}{2n} v_n$ and hence $q\perp e, v_n$ for all $n$, which proves that $\displaystyle q\perp e + \sum_{n=1}^\infty v_n =\textbf{1}$ in $B(H_2)$, which is impossible. Therefore, case $(1)$ holds, and thus $q\leq e$. Since $e$ is a minimal projection in $K(H_2)^{**}=B(H_2)$, we deduce from the minimality of $q$ that $e=q\perp b$.\smallskip

We have shown that for every minimal projection $e$ in $B(H_2)$ with $e\leq \textbf{1}-p$ we have $e\perp b$, and then $\textbf{1}-p \perp b$ (equivalently, $ p b = b p =b$).\smallskip

The above arguments show that $b,p\in p K(H_2) p \cong M_n (\mathbb{C})$. Furthermore, every $x\in Sph^+_{p K(H_2) p}(a)$ lies in $Sph^+_{K(H_2)}(a)$ and hence $\| b-x\|=1$, therefore $b$ lies in $Sph^+_{p K(H_2) p}(Sph^+_{p K(H_2) p}(p)).$ It follows from \cite[final paragraph of the proof of Claim 1]{Nagy2017} (see also \eqref{eq bi sphere around a point}) that $Sph^+_{p K(H_2) p}(Sph^+_{p K(H_2) p}(p))=\{p\}$, and hence $b=p$. Therefore, $Sph^+_{K(H_2)} \left( Sph^+_{K(H_2)}(p) \right) =\{p\}$. \smallskip

The implication $(b)\Rightarrow (a)$ follows from Proposition \ref{p necessary condition for projection in terms of the sphere around a positive element}.
\end{proof}

Many consequences can be expected from the characterizations established in Theorem \ref{t characterization of projection in terms of the sphere around a positive element} and Corollary \ref{c characterization of projection in terms of the sphere around a positive element B(H)}. We shall conclude this note with a first application. For a C$^*$-algebra $A$, let $\mathcal{P}roj(A)^*$ denote the set of all non-zero projections in $A$. The next result is an infinite-dimensional version of \cite[Claim 1]{Nagy2017} which proves one of the conjectures posed at the end of the just quoted paper.

\begin{corollary}\label{c first consequence} Let $\Delta : S(M^+)\to S(N^+)$ be a surjective isometry, where $M$ and $N$ are atomic von Neumann algebras. Then $\Delta$ maps $\mathcal{P}roj(M)^*$ onto $\mathcal{P}roj(N)^*$, and the restriction $\Delta|_{\mathcal{P}roj(M)^*} : \mathcal{P}roj(M)^*\to \mathcal{P}roj(N)^*$ is a surjective isometry.
\end{corollary}

\begin{proof} Let $p$ be a non-zero projection in $M$. Applying Theorem \ref{t characterization of projection in terms of the sphere around a positive element} we have $Sph^+_{M} \left( Sph^+_{M}(p) \right) =\{p\}$. Since $\Delta$ is a surjective isometry, the sphere around a set $E\subset S(M^+)$, $Sph^+_{M}(E),$ is always preserved by $\Delta$, that is, $\Delta\left( Sph^+_{M}(E) \right) = Sph^+_{N}(\Delta(E))$. We consequently have $$\{\Delta(p)\}=\Delta(\{p\})=\Delta\left(Sph^+_{M} \left( Sph^+_{M}(p) \right) \right) = Sph^+_{N} \left( Sph^+_{N}(\Delta(p)) \right),$$ and a new application of Theorem \ref{t characterization of projection in terms of the sphere around a positive element} assures that $\Delta(p)$ is a projection in $N$.\smallskip

We have shown that $\Delta (\mathcal{P}roj(M)^*) \subseteq \mathcal{P}roj(N)^*$. Since $\Delta^{-1}$ also is a surjective isometry, we get $\Delta (\mathcal{P}roj(M)^*) = \mathcal{P}roj(N)^*$. Clearly $\Delta|_{\mathcal{P}roj(M)^*} : \mathcal{P}roj(M)^*\to \mathcal{P}roj(N)^*$ is a surjective isometry.
\end{proof}

When in the previous proof we replace Theorem \ref{t characterization of projection in terms of the sphere around a positive element} with Theorem \ref{t characterization of projection in terms of the sphere around a positive element k(H)} the same arguments are valid to prove the following:

\begin{corollary}\label{c first consequence K(H)} Let $H_2$ and $H_3$ be separable complex Hilbert spaces, and let us assume that $\Delta : S(K(H_2)^+)\to S(K(H_3)^+)$ is a surjective isometry. Then $\Delta$ maps $\mathcal{P}roj(K(H_2))^*$ to $\mathcal{P}roj(K(H_3))^*$, and the restriction $$\Delta|_{\mathcal{P}roj(K(H_2))^*} : \mathcal{P}roj(K(H_2))^*\to \mathcal{P}roj(K(H_3))^*$$ is a surjective isometry.$\hfill\Box$
\end{corollary}

Another result established by G. Nagy in \cite{Nagy2017} asserts that for a finite-dimensional complex Hilbert space $H$, the equality $$ Sph^+_{B(H)} \left( Sph^+_{B(H)}(a) \right) =\left\{ b\in S(B(H)^+) : \begin{array}{c}
                                                                                           \hbox{Fix}(a) \subseteq \hbox{Fix} (b), \\
                                                                                            \hbox{ and } \ker(a)\subseteq \ker(b)
                                                                                         \end{array} \right\}$$
holds for every element $a$ in $S(B(H)^+)$ (see the beginning of the proof of \cite[Claim 1]{Nagy2017}). Our next result is an abstract version of Nagy's result to the space of compact operators.

\begin{theorem}\label{t bi spherical set K(H)} Let $H_2$ be a separable infinite-dimensional complex Hilbert space. Then the identity $$Sph^+_{K(H_2)} \left( Sph^+_{K(H_2)}(a) \right) =\left\{ b\in S(K(H_2)^+) : \!\! \begin{array}{c}
    s_{_{K(H_2)}} (a) \leq s_{_{K(H_2)}} (b),  \hbox{ and }\\
     \textbf{1}-r_{_{B(H_2)}}(a)\leq \textbf{1}-r_{_{B(H_2)}}(b)
  \end{array}\!\!
 \right\},$$ holds for every $a$ in the unit sphere of $K(H_2)^+$.
\end{theorem}

\begin{proof}($\supseteq$) We recall that, for each $b\in S(K(H_2)^+)$ we have $s_{_{K(H_2)}} (b) =p_{b}\in K(H_2)$. Let $b\in S(K(H_2)^+)$ with $s_{_{K(H_2)}} (a) \leq s_{_{K(H_2)}} (b)$, and $\textbf{1}-r_{_{B(H_2)}}(a)\leq \textbf{1}-r_{_{B(H_2)}}(b)$. We pick an arbitrary $c\in Sph^+_{K(H_2)}(a)$. Since $\| a-c\|=1$, Lemma \ref{l when distance is exactly one} implies the existence of a minimal projection $e$ in $B(H_2)$ such that one of the following statements holds:
\begin{enumerate}[$(a)$]\item $e\leq a$ and $e\perp c$ in $K(H_2)^{**}= B(H_2)$;
\item $e\leq c$ and $e\perp a$ in $K(H_2)^{**}=B(H_2)$.
\end{enumerate}

In case $(a)$, we have $e\leq s_{_{K(H_2)}}(a) \leq s_{_{K(H_2)}}(b)$ and $e\perp c$. Lemma \ref{l when distance is exactly one} implies that $\| c-b \| =1$.\smallskip

In case $(b)$, the condition $e\perp a$ implies that $e\leq \textbf{1}-r_{_{B(H_2)}}(a) \leq \textbf{1}-r_{_{B(H_2)}}(b)$, and thus $e\perp b$. Since $e\leq c$, Lemma \ref{l when distance is exactly one} assures that $\|c-b\|=1$.\smallskip

We have shown that $\|c-b\|=1$ for all $c \in Sph^+_{K(H_2)}(a)$, and thus $b$ lies in $Sph^+_{K(H_2)}(Sph^+_{K(H_2)}(a)).$\smallskip

($\subseteq$) Let us take $b\in Sph^+_{K(H_2)} \left( Sph^+_{K(H_2)}(a) \right)$.\smallskip

We shall first prove that $\textbf{1}-r_{_{B(H_2)}}(a)\leq \textbf{1}-r_{_{B(H_2)}}(b)$. If $\textbf{1}-r_{_{B(H_2)}}(a) =0$ there is nothing to prove. Otherwise, let $e$ be a minimal projection in $K(H_2)$ with $e\leq \textbf{1}-r_{_{B(H_2)}}(a)$. Let $(e_n)$ be a maximal family of mutually orthogonal minimal projections in $K(H_2)$ such that $\displaystyle \textbf{1}-e = \sum_{n=1}^{\infty} e_n$ (here we apply that $H_2$ is separable). Since $\displaystyle \left\| e + \sum_{n=1}^{\infty} \frac{1}{2n} e_n -a\right\| =1,$ and $\displaystyle e + \sum_{n=1}^{\infty} \frac{1}{2n} e_n \in K(H_2)$, we deduce that $\displaystyle \left\| e + \sum_{n=1}^{\infty} \frac{1}{2n} e_n -b\right\| =1$. Lemma \ref{l when distance is exactly one} proves the existence of a minimal projection $q\in B(H_2)$ such that one of the next statements holds: \begin{enumerate}[$(a)$]\item $\displaystyle q\leq e + \sum_{n=1}^{\infty} \frac{1}{2n} e_n$ and $q\perp b$ in $B(H_2)$;
\item $q\leq b$ and $\displaystyle q\perp e + \sum_{n=1}^{\infty} \frac{1}{2n} e_n$ in $B(H_2)$.
\end{enumerate} We claim that case $(b)$ is impossible. Indeed, $\displaystyle q\perp e + \sum_{n=1}^{\infty} \frac{1}{2n} e_n$ is equivalent to $\displaystyle q\perp r_{_{B(H_2)}} \left(e + \sum_{n=1}^{\infty} \frac{1}{2n} e_n\right) =\textbf{1},$ which is impossible. Therefore, only case $(a)$ holds, and by the minimality of $q$, $q$ coincides with $e,$ and $e=q\perp b$ assures that $q=e\leq \textbf{1}-r_{_{B(H_2)}}(b)$.\smallskip

We have shown that for every minimal projection $e$ in $B(H_2)$ with $e\leq \textbf{1}-r_{_{B(H_2)}}(a)$ we have $q\leq \textbf{1}-r_{_{B(H_2)}}(b)$. Since in $B(H_2)$ every projection is the least upper bound of all minimal projections smaller than or equal to it, we deduce that  $$\textbf{1}-r_{_{B(H_2)}}(a)\leq \textbf{1}-r_{_{B(H_2)}}(b).$$\smallskip

Our next goal is to show that $s_{_{K(H_2)}} (a) \leq s_{_{K(H_2)}} (b)$. If $r_{_{B(H_2)}}(a)- s_{_{B(H_2)}}(a)=0$, we have $ s_{_{K(H_2)}}(a)=a=r_{_{B(H_2)}}(a)\geq  r_{_{B(H_2)}}(b)\geq  s_{_{B(H_2)}}(b)$. In particular, $a$ is a projection in $K(H_2)$. We shall prove that $b$ is a projection and $a=b$. Let $\sigma(b)$ be the spectrum of $b$, let $\mathcal{C}$ denote the C$^*$-subalgebra of $K(H_2)$ generated by $b$ and $a=r_{_{K(H_2)}}(a)$, and let us identify $\mathcal{C}$ with $C(\sigma(b))$ and $b$ with the identity function on $\sigma(b)$. If there exists $t_0\in \sigma(b)\cap (0,1)$, then the function \begin{equation}\label{eq function c bis} c(t)= \left\{\begin{array}{cc}
                                   0 & \hbox{ if } t\in \sigma(b)\cap [0,t_0]; \\
                                   \frac{1+t_0}{1-t_0} (t-t_0) & \hbox{ if } t\in \sigma(b)\cap [0,t_0]; \\
                                   t & \hbox{ if } t\in \sigma(b)\cap [\frac{1+t_0}{2},1],
                                 \end{array}
 \right.
 \end{equation} defines a positive, norm-one element in $c\in C(\sigma(b))\subset K(H_2)$ such that $\| a -c\| =1$ and $\| b-c\| <1$. This contradicts that $b \in Sph^+_{K(H_2)} \left( Sph^+_{K(H_2)}(a) \right)$. Therefore, $\sigma(b) \subseteq\{0,1\}$, and hence $b$ is a projection. If $s_{_{B(H_2)}}(b) = b < s_{_{K(H_2)}}(a)=a$, we get $\| b - s_{_{K(H_2)}}(b)\|=0,$ and $\|a -b\| = \|a - s_{_{K(H_2)}}(b)\| = \|s_{_{K(H_2)}}(a) - s_{_{K(H_2)}}(b)\| =1$, contradicting that $b \in Sph^+_{K(H_2)} \left( Sph^+_{K(H_2)}(a) \right)$. Therefore $a=b$ is a projection and $s_{_{K(H_2)}}(b)= b=a= s_{_{K(H_2)}}(a)$.\smallskip

We assume next that $r_{_{B(H_2)}}(a)- s_{_{K(H_2)}}(a)\neq 0$. We first prove the following \textbf{Property~$(\checkmark.1)$}:
for each pair of minimal projections $v,q\in B(H_2)$ with $v\leq s_{_{K(H_2)}} (a)$ and $q\leq r_{_{B(H_2)}}(a)- s_{_{K(H_2)}}(a)$ one of the following statements holds:
\begin{enumerate}[$(1)$]\item $q\perp b$, or equivalently, $q\leq \textbf{1}-r_{_{B(H_2)}}(b)$;
\item $v\leq s_{_{B(H_2)}}(b) \leq b$.
\end{enumerate}

To prove the property, we consider a family $(v_n)$ of mutually orthogonal minimal projections in $K(H_2)$ satisfying $\displaystyle \textbf{1}-v-q =\sum_{n=1}^{\infty} v_n$, and the element $\displaystyle q+ \sum_{n=1}^{\infty} \frac{1}{2n} v_n \in S(K(H_2)^+)$. Clearly, $v$ is a minimal projection in $B(H_2)$ satisfying $v\leq a$ and $v\perp q, \textbf{1}-v$, and hence $v\perp \displaystyle q+ \sum_{n=1}^{\infty} \frac{1}{2n} v_n $. Lemma \ref{l when distance is exactly one} assures that $\displaystyle \left\| a-\left(\displaystyle q+ \sum_{n=1}^{\infty} \frac{1}{2n} v_n \right)\right\|=1,$ and by hypothesis $\displaystyle \left\| b-\left(\displaystyle q+ \sum_{n=1}^{\infty} \frac{1}{2n} v_n \right)\right\|=1$.  A new application of Lemma \ref{l when distance is exactly one} assures the existence of a minimal projection $e\in B(H_2)$ such that one of the following statements holds:\begin{enumerate}[$(a)$]\item $e\leq b$ and $e\perp \displaystyle q+ \sum_{n=1}^{\infty} \frac{1}{2n} v_n$ in $B(H_2)$;
\item $e\leq \displaystyle q+ \sum_{n=1}^{\infty} \frac{1}{2n} v_n$ and $e\perp b$ in $B(H_2)$.
\end{enumerate}
In the second case $e= q \perp b,$ equivalently, $q\leq \textbf{1}-r_{_{B(H_2)}}(b)$. In the first case $e\leq b\leq r_{_{B(H_2)}}(b) \leq r_{_{B(H_2)}}(a)$, and $e\perp  q, \textbf{1}-v.$ Since $e\leq  r_{_{B(H_2)}}(a)$ and $r_{_{B(H_2)}}(a) = (r_{_{B(H_2)}}(a)-v) +v$, we deduce that $e\leq v$. The minimality of $e$ and $v$ proves that $e=v \leq b$, and thus $v\leq s_{_{B(H_2)}}(b) \leq b$. This finishes the proof of \emph{Property~$(\checkmark.1)$}.\smallskip

We discuss now the following dichotomy:
\begin{enumerate}[$\bullet$]\item There exists a minimal projection $v$ in $B(H_2)$ with $v\leq s_{_{K(H_2)}} (a)$ and $v\nleq  s_{_{K(H_2)}} (b)$;
\item For every minimal projection $v$ in $B(H_2)$ with $v\leq s_{_{K(H_2)}} (a)$ we have $v\leq  s_{_{K(H_2)}} (b)$.
\end{enumerate}

In the first case, let $v$ be a minimal projection in $K(H_2)$ with $v\leq s_{_{K(H_2)}} (a)$ and $v\nleq  s_{_{K(H_2)}} (b)$. \emph{Property~$(\checkmark.1)$} implies that for every minimal projection $q\in B(H_2)$ with $q\leq r_{_{B(H_2)}}(a)- s_{_{K(H_2)}}(a)$ we have $q\leq \textbf{1}-r_{_{B(H_2)}} (b)$. This proves that $$r_{_{B(H_2)}}(a)- s_{_{K(H_2)}}(a)\leq \textbf{1}-r_{_{B(H_2)}} (b).$$ We have therefore shown that $$\textbf{1}- s_{_{K(H_2)}} (a)=(\textbf{1}-r_{_{B(H_2)}} (a)) + (r_{_{B(H_2)}} (a)- s_{_{K(H_2)}} (a)) \leq  \textbf{1}-r_{_{B(H_2)}} (b),$$ and thus $r_{_{B(H_2)}} (b)\leq s_{_{K(H_2)}} (a).$ In this case we have $0\leq b \leq r_{_{B(H_2)}} (b)\leq s_{_{K(H_2)}} (a)$, and then $ a b = ba = b$. If $\sigma(b)\cap (0,1)\neq \emptyset$, by considering the C$^*$-subalgebra of $K(H_2)$ generated by $b$, and the definition in \eqref{eq function c bis}, we can find an element $c$ in $S(K(H_2)^+)$ such that $\|a-c\|= 1$ and $\|b-c\|<1$, contradicting that $b \in Sph^+_{K(H_2)} \left( Sph^+_{K(H_2)}(a) \right)$. Therefore $\sigma(b)\subseteq \{0,1\},$ and hence $b$ is a projection with $b\leq s_{_{K(H_2)}} (a)$. If $b< s_{_{K(H_2)}} (a)$, we have $\|b-b\|=0$ and $\|a-b\|=1$ contradicting, again, that $b \in Sph^+_{K(H_2)} \left( Sph^+_{K(H_2)}(a) \right)$. We have shown that in this case $b= s_{_{K(H_2)}} (b) = s_{_{K(H_2)}} (a)$.\smallskip

In the second case of the above dichotomy, having in mind that $s_{_{K(H_2)}} (a)$ can be written as a finite sum of mutually orthogonal minimal projections in $K(H_2)$, we have $s_{_{K(H_2)}} (a)\leq  s_{_{K(H_2)}} (b)$ as desired.
\end{proof}

\begin{remark}\label{remark characterization proj in K(H) is a consequence of previous Th} Let us remark that Theorem \ref{t characterization of projection in terms of the sphere around a positive element k(H)} can be derived as a straight consequence of our previous Theorem \ref{t bi spherical set K(H)}. Namely, let $H_2$ be a separable complex Hilbert space, and let $a$ be an element in $S(K(H_2)^+)$. Applying Theorem \ref{t bi spherical set K(H)} we get $$Sph^+_{K(H_2)} \left( Sph^+_{K(H_2)}(a) \right) =\left\{ b\in S(K(H_2)^+) : \!\! \begin{array}{c}
    s_{_{K(H_2)}} (a) \leq s_{_{K(H_2)}} (b),  \hbox{ and }\\
     \textbf{1}-r_{_{B(H_2)}}(a)\leq \textbf{1}-r_{_{B(H_2)}}(b)
  \end{array}\!\!
 \right\}.$$ If $a$ is a projection, then $s_{_{K(H_2)}} (a) = r_{_{B(H_2)}} (a)= a$ and hence $$Sph^+_{K(H_2)} \left( Sph^+_{K(H_2)}(a) \right) = \{a\}.$$ If, on the other hand, $Sph^+_{K(H_2)} \left( Sph^+_{K(H_2)}(a) \right) = \{a\}$, having in mind that $s_{_{K(H_2)}} (a)$ belongs to $S(K(H_2)^+)$, and $s_{_{K(H_2)}} (a) \leq r_{_{B(H_2)}} (a)$, we deduce that $s_{_{K(H_2)}} (a) $ lies in the set $Sph^+_{K(H_2)} \left( Sph^+_{K(H_2)}(a) \right) = \{a\}$, and hence $s_{_{K(H_2)}} (a) = a$ is a projection.
\end{remark}

\medskip\medskip

\textbf{Acknowledgements} Author partially supported by the Spanish Ministry of Economy and Competitiveness (MINECO) and European Regional Development Fund project no. MTM2014-58984-P and Junta de Andaluc\'{\i}a grant FQM375.

\bibliographystyle{amsplain}

\begin{thebibliography}{99}

\bibitem{FerGarPeVill17} F.J. Fern\'andez-Polo, J.J. Garc{\'e}s, A.M. Peralta, I. Villanueva, Tingley's problem for spaces of trace class operators,  \emph{Linear Algebra Appl.} \textbf{529}, 294-323 (2017).

\bibitem{FerJorPer2018} F.J. Fernández-Polo, E. Jord{\'a}, A.M. Peralta, Tingley's problem for $p$-Schatten von Neumann classes for $2<p<\infty$, preprint 2018. arXiv:1803.00763v1


\bibitem{FerPe17c} F.J. Fern\'andez-Polo, A.M. Peralta, Tingley's problem through the facial structure of an atomic JBW$^*$-triple, \emph{J. Math. Anal. Appl.} \textbf{455}, 750-760 (2017).

\bibitem{FerPe17b} F.J. Fern\'andez-Polo, A.M. Peralta, On the extension of isometries between the unit spheres of a C$^*$-algebra and $B(H)$,  \emph{Trans. Amer. Math. Soc.} \textbf{5}, 63-80 (2018).

\bibitem{FerPe18} F.J. Fern\'andez-Polo, A.M. Peralta, Partial Isometries: a survey, \emph{Adv. Oper. Theory} \textbf{3}, no. 1, 87-128 (2018).

\bibitem{FerPe17d} F.J. Fern\'andez-Polo, A.M. Peralta, On the extension of isometries between the unit spheres of von Neumann algebras, preprint 2017. arXiv:1709.08529v1

\bibitem{MolNag2012} L. Moln{\'a}r, G. Nagy, Isometries and relative entropy preserving maps on density operators, \emph{Linear Multilinear Algebra} \textbf{60}, 93-108 (2012).

\bibitem{MolTim2003} L. Moln{\'a}r, W. Timmermann, Isometries of quantum states, \emph{J. Phys. A: Math. Gen.} \textbf{36}, 267-273 (2003).

\bibitem{Mori2017} M. Mori, Tingley's problem through the facial structure of operator algebras, preprint 2017. arXiv:1712.09192v1

\bibitem{Nagy2013} G. Nagy, Isometries on positive operators of unit norm, \emph{Publ. Math. Debrecen} \textbf{82}, 183-192 (2013).

\bibitem{Nagy2017} G. Nagy, Isometries of spaces of normalized positive operators under the operator norm, \emph{Publ. Math. Debrecen} \textbf{92}, 243-254 (2018).

\bibitem{Ped} G.K. Pedersen, \emph{C$^*$-algebras and their automorphism groups},
London Mathematical Society Monographs Vol. 14, Academic Press, London, 1979.

\bibitem{Pe2018} A.M. Peralta, A survey on Tingley's problem for operator algebras, to appear in \emph{Acta Sci. Math. Szeged}. arXiv:1801.02473v1

\bibitem{PeTan16} A.M. Peralta, R. Tanaka, A solution to Tingley's problem for isometries between the unit spheres of compact C$^*$-algebras and JB$^*$-triples, to appear in \emph{Sci. China Math.} arXiv:1608.06327v1.

\bibitem{S} S. Sakai, \emph{C$^*$-algebras and $W^*$-algebras}, Springer, Berlin, 1971.

\bibitem{Tak} M. Takesaki, \newblock {\em Theory of operator algebras I},\newblock Springer, New York, 2003.

\bibitem{Tan2017} R. Tanaka, Spherical isometries of finite dimensional C$^*$-algebras, \emph{J. Math. Anal. Appl.} \textbf{445}, no. 1, 337-341 (2017).

\bibitem{Tan2017b} R. Tanaka, Tingley's problem on finite von Neumann algebras, \emph{J. Math. Anal. Appl.} \textbf{451}, 319-326 (2017).

\end{thebibliography}

\end{document}